\documentclass{amsart}
\usepackage{amsfonts}

\newtheorem{theorem}{Theorem}
\theoremstyle{plain}

\newtheorem{corollary}{Corollary}

\newtheorem{definition}{Definition}

\newtheorem{remark}{Remark}

\numberwithin{equation}{section}
\input{tcilatex}

\begin{document}
\title[The Hermite -Hadamard's inequalities for $\varphi -r-convex$ functions%
]{On The Hadamard Type Integral Inequalities Involving Several $\varphi -r-$%
Convex Functions }
\author{Mehmet Zeki SARIKAYA}
\address{Department of Mathematics, \ Faculty of Science and Arts, D\"{u}zce
University, D\"{u}zce-TURKEY}
\email{sarikayamz@gmail.com}
\author{Hatice YALDIZ}
\email{yaldizhatice@gmail.com}
\author{Hakan Bozkurt}
\email{insedi@yahoo.com}
\subjclass[2000]{ 26D07, 26D10, 26D99 }
\keywords{Hermite-Hadamard's inequality, log-convex functions, Logarithmic
mean, Cauchy inequality, Minkowski inequality, Young inequality, H\"{o}%
lder's inequality.}

\begin{abstract}
In this paper, new integral inequalities of Hadamard type involving several
differentiable $\varphi -r-$convex functions are given.
\end{abstract}

\maketitle

\section{Introduction}

It is well known that if $f$ is a convex function on the interval $I=\left[
a,b\right] $ with $a<b$, then

\begin{equation*}
f\left( \frac{a+b}{2}\right) \leq \frac{1}{b-a}\int\limits_{a}^{b}f\left(
x\right) dx\leq \frac{f\left( a\right) +f\left( b\right) }{2},
\end{equation*}%
which is known as the Hermite-Hadamard inequality for the convex functions.

In \cite{pec2} Pearce et. al. generalized this inequality to $r$-convex
positive function $f$ which defined on an interval $[a,b]$, for all $x,y\in
\lbrack a,b]$ and $t\in \lbrack 0,1]$%
\begin{equation*}
f\left( tx+(1-t)y\right) \leq \left\{ 
\begin{array}{ll}
\left( t\left[ f\left( x\right) \right] ^{r}+\left( 1-t\right) \left[
f\left( y\right) \right] ^{r}\right) ^{\frac{1}{r}}, & \text{if }r\neq 0 \\ 
\left[ f\left( x\right) \right] ^{t}\left[ f\left( y\right) \right] ^{1-t},
& \text{if }r=0.%
\end{array}%
\right.
\end{equation*}%
We have that $0$-convex functions are simply $\log $-convex functions and $1$%
-convex functions are ordinary convex functions.

Recently, the generalizations of the Hermite-Hadamard's inequality to the
integral power mean of a positive convex function on an interval $[a,b]$,
and to that of a positive $r$-convex function on an interval $[a,b]$ are
obtained by Pearce and Pecaric, and others (see \cite{pec1}-\cite{ngoc}).

For some results related to this classical results, (see\cite{dragomir1},%
\cite{dragomir2},\cite{pach},\cite{pec1}$)$ and the references therein.
Dragomir and Mond \cite{dragomir1} proved the following Hermite-Hadamard
type inequalities for the $log$-convex functions:

\begin{eqnarray}
f\left( \frac{a+b}{2}\right) &\leq &\exp \left[ \frac{1}{b-a}%
\int\limits_{a}^{b}\ln \left[ f\left( x\right) \right] dx\right]  \label{z2}
\\
&\leq &\frac{1}{b-a}\int\limits_{a}^{b}G\left( f\left( x\right) ,f\left(
a+b-x\right) \right) dx  \notag \\
&\leq &\frac{1}{b-a}\int\limits_{a}^{b}f\left( x\right) dx  \notag \\
&\leq &L\left( f\left( a\right) ,f\left( b\right) \right)  \notag \\
&\leq &\frac{f\left( a\right) +f\left( b\right) }{2},  \notag
\end{eqnarray}%
where $G\left( p,q\right) =\sqrt{pq}$ is the geometric mean and $L\left(
p,q\right) =\frac{p-q}{\ln p-\ln q}$ $\left( p\neq q\right) $ is the
logarithmic mean of the positive real numbers $p,q$ $\left( \text{for }p=q,%
\text{ we put }L\left( p,q\right) =p\right) $.

This paper, except for the introduction, is divided into two sections. In
Section 1, we give the some definitions of $\varphi -$convex functions given
by Noor in \cite{noor1} and \cite{noor5} and we will give a new definition.
By using the new definition is defined in Section 1, we will give the proof
of main theorems in Section 2.

\section{Definitions}

Let $K$ be a nonempty closed set in $%
%TCIMACRO{\U{211d} }%
%BeginExpansion
\mathbb{R}
%EndExpansion
^{n}$. Let $f,\varphi :K\rightarrow 
%TCIMACRO{\U{211d} }%
%BeginExpansion
\mathbb{R}
%EndExpansion
$ be continuous functions. First of all, we recall the following well know
results and concepts, which are mainly due to Noor and Noor \cite{noor5} and
Noor \cite{noor1}. In \cite{noor1} and \cite{noor5}, the following new class
of functions are defined by Noor:

\begin{definition}
\label{d1} Let $u\in K$. Then the set $K$ is said to be $\varphi -convex$ at 
$u$ with respect to $\varphi $, if
\end{definition}

\begin{equation*}
u+te^{i\varphi }\left( v-u\right) \in K,\text{ }\forall u,v\in K,\text{ }%
t\in \left[ 0,1\right] .
\end{equation*}

\begin{remark}
\label{r1} We would like to mention that the (\ref{d1}) of a $\varphi
-convex $ set has a clear geometric interpretation. This definition
essentially says that there is a path starting from a point $u$ which is
contained in $K$. We do not require that the point $v$ should be one of the
end points of the path. This observation plays an important role in our
analysis. Note that, if we demand that $v$ should be an end point of the
path for every pair of points, $u,v\in K$, then $e^{i\varphi }\left(
v-u\right) =v-u$ if and only if, $\varphi =0$, and consequently $\varphi
-convexity$ reduces to convexity. Thus, it is true that every convex set is
also an $\varphi -convex $ set, but the converse is not necessarily true,
see \cite{noor1},\cite{noor5} and the references therein.
\end{remark}

\begin{definition}
\label{d2} The function $f$ on the $\varphi -convex$ set $K$ is said to be $%
\varphi -convex$ with respect to $\varphi $, if%
\begin{equation*}
f\left( u+te^{i\varphi }\left( v-u\right) \right) \leq \left( 1-t\right)
f\left( u\right) +tf\left( v\right) ,\text{ }\forall u,v\in K,\text{ }t\in %
\left[ 0,1\right] .
\end{equation*}
\end{definition}

The function $f$ is said to be $\varphi -concave$ if and only if $-f$ is $%
\varphi -convex$. Note that every convex function is a $\varphi -convex$
function, but the converse is not true.

\begin{definition}
\label{d3} The function $f$ on the $\varphi -convex$ set $K$ is said to be
logarithmic $\varphi -convex$ with respect to $\varphi $\ , such that%
\begin{equation*}
f\left( u+te^{i\varphi }\left( v-u\right) \right) \leq \left( f\left(
u\right) \right) ^{1-t}\left( f\left( v\right) \right) ^{t},\text{ }u,v\in K,%
\text{ }t\in \left[ 0,1\right] ,
\end{equation*}
\end{definition}

where $f\left( .\right) >0\left( \text{\cite{noor1},\cite{noor3},\cite{noor5}%
}\right) .$

Now, we will define a new definition for $\varphi -r-convex$ fonctions as
follows:

\begin{definition}
\label{d4} The positive function $f$ on the $\varphi -r-convex$ set $K$ is
said to be $\varphi -r-convex$ with respect to $\varphi $, if%
\begin{equation*}
f\left( u+te^{i\varphi }\left( v-u\right) \right) \leq \left\{ 
\begin{array}{ll}
\left( \left( 1-t\right) \left[ f\left( u\right) \right] ^{r}+t\left[
f\left( v\right) \right] ^{r}\right) ^{\frac{1}{r}}, & r\neq 0 \\ 
\left[ f\left( u\right) \right] ^{1-t}\left[ f\left( v\right) \right] ^{t},
& r=0.%
\end{array}%
\right. 
\end{equation*}
\end{definition}

We have that $\varphi -0-$convex functions are simply logarithmic $\varphi -$%
convex functions and $\varphi -1-$convex functions are $\varphi -$convex
functions.

From the above definitions, we have%
\begin{eqnarray*}
f\left( u+te^{i\varphi }\left( v-u\right) \right) &\leq &\left( f\left(
u\right) \right) ^{1-t}\left( f\left( v\right) \right) ^{t} \\
&\leq &\left( 1-t\right) f\left( u\right) +tf\left( v\right) \\
&\leq &\max \left\{ f\left( u\right) ,f\left( v\right) \right\} .
\end{eqnarray*}%
In \cite{noor3}, Noor established following theorem for $\varphi -$convex
functions:

\begin{theorem}
\label{tt1} Let $f:K=\left[ a,a+e^{i\varphi }\left( b-a\right) \right]
\rightarrow \left( 0,\infty \right) $ be a $\varphi -convex$\ function on
the interval of real numbers $K^{0}$ (the interior of $K$) and $a,b\in K^{0}$
with $a<a+e^{i\varphi }\left( b-a\right) $ and $0\leq \varphi \leq \frac{\pi 
}{2}$. Then%
\begin{equation}
f\left( \frac{2a+e^{i\varphi }\left( b-a\right) }{2}\right) \leq \frac{1}{%
e^{i\varphi }\left( b-a\right) }\int\limits_{a}^{a+e^{i\varphi }\left(
b-a\right) }f\left( x\right) dx\leq \frac{f\left( a\right) +f\left( b\right) 
}{2}.  \label{z3}
\end{equation}
\end{theorem}

The main purpose of this note is to establish new integral inequalities
Hadamard type involving product of two $\varphi -r-convex$ fonctions. Two
refinements of Hadamard's integral inequality for $r$-convex functions
recently established by Ngoc et. al. are shown to be recaptured as special
instances. The method employed in our analysis is based on the basic
properties of logarithms and the application of the well known H\"{o}lder's
integral inequality and Minkowski's integral inequality.

\section{Main Results}

Now, we start with the following our main theorem.

\begin{theorem}
\label{tt2} Let $f:K=\left[ a,a+e^{i\varphi }\left( b-a\right) \right]
\rightarrow \left( 0,\infty \right) $ be $\varphi -r-convex$ functions on
the interval of real numbers $K^{0}$ (the interior of $K$) and $a,b\in K^{0}$
with $a<a+e^{i\varphi }\left( b-a\right) $ and $0\leq \varphi \leq \frac{\pi 
}{2}$. Then
\end{theorem}

\begin{equation}
\frac{1}{e^{i\varphi }\left( b-a\right) }\int\limits_{a}^{a+e^{i\varphi
}\left( b-a\right) }f\left( x\right) dx\leq \left( \frac{r}{r+1}\right) ^{%
\frac{1}{r}}\left( \left[ f\left( a\right) \right] ^{r}+\left[ f\left(
b\right) \right] ^{r}\right) ^{\frac{1}{r}}.  \label{z4}
\end{equation}

\begin{proof}
Since $f$ is $\varphi -r-convex$ function and $r\neq 0$, we have%
\begin{equation*}
f\left( u+te^{i\varphi }\left( v-u\right) \right) \leq \left( \left(
1-t\right) \left[ f\left( u\right) \right] ^{r}+t\left[ f\left( v\right) %
\right] ^{r}\right) ^{\frac{1}{r}},r\neq 0
\end{equation*}%
for all $t\in \left[ 0,1\right] $. It is easy to observe that%
\begin{eqnarray}
\frac{1}{e^{i\varphi }\left( b-a\right) }\int\limits_{a}^{a+e^{i\varphi
}\left( b-a\right) }f\left( x\right) dx &=&\int\limits_{0}^{1}f\left(
a+te^{i\varphi }\left( b-a\right) \right) dt  \label{z5} \\
&\leq &\int\limits_{0}^{1}\left( \left( 1-t\right) \left[ f\left( a\right) %
\right] ^{r}+t\left[ f\left( b\right) \right] ^{r}\right) ^{\frac{1}{r}}dt. 
\notag
\end{eqnarray}%
Using Minkowski's inequality (\ref{z5}), we have%
\begin{eqnarray*}
\int\limits_{0}^{1}\left( \left( 1-t\right) \left[ f\left( a\right) \right]
^{r}+t\left[ f\left( b\right) \right] ^{r}\right) ^{\frac{1}{r}}dt &\leq &%
\left[ \left( \int\limits_{0}^{1}\left( 1-t\right) ^{\frac{1}{r}}f\left(
a\right) dt\right) ^{r}+\left( \int\limits_{0}^{1}t^{\frac{1}{r}}f\left(
b\right) dt\right) ^{r}\right] ^{\frac{1}{r}} \\
&=&\left( \left( \frac{r}{r+1}\right) \left[ f\left( a\right) \right]
^{r}+\left( \frac{r}{r+1}\right) \left[ f\left( b\right) \right] ^{r}\right)
^{\frac{1}{r}} \\
&=&\left( \frac{r}{r+1}\right) ^{\frac{1}{r}}\left( \left[ f\left( a\right) %
\right] ^{r}+\left[ f\left( b\right) \right] ^{r}\right) ^{\frac{1}{r}}.
\end{eqnarray*}%
Thus, it is the required inequality in (\ref{z4}). This proof is complete.
\end{proof}

\begin{corollary}
\label{c1} Under the asumptions of Theorem \ref{tt2} with $r=1$, the
following inequality holds:
\end{corollary}

\begin{equation*}
\frac{1}{e^{i\varphi }\left( b-a\right) }\int\limits_{a}^{a+e^{i\varphi
}\left( b-a\right) }f\left( x\right) dx\leq \frac{f\left( a\right) +f\left(
b\right) }{2}.
\end{equation*}

\begin{theorem}
\label{tt4} Let $f,g:K=\left[ a,a+e^{i\varphi }\left( b-a\right) \right]
\rightarrow \left( 0,\infty \right) $ be $\varphi -r-convex$ and $\varphi
-s-convex$\ functions on the interval of real numbers $K^{0}$ (the interior
of $K$) and $a,b\in K^{0}$ with $a<a+e^{i\varphi }\left( b-a\right) $ and $%
0\leq \varphi \leq \frac{\pi }{2}$. Then%
\begin{eqnarray}
&&\frac{2}{e^{i\varphi }\left( b-a\right) }\int\limits_{a}^{a+e^{i\varphi
}\left( b-a\right) }f\left( x\right) g\left( x\right) dx  \notag \\
&&  \label{16} \\
&\leq &\left( \frac{r}{r+2}\right) \left( \left[ f\left( a\right) \right]
^{r}+\left[ f\left( b\right) \right] ^{r}\right) ^{\frac{2}{r}}+\left( \frac{%
s}{s+2}\right) \left( \left[ g\left( a\right) \right] ^{s}+\left[ g\left(
b\right) \right] ^{s}\right) ^{\frac{2}{s}}  \notag
\end{eqnarray}%
and%
\begin{eqnarray}
&&\frac{1}{e^{i\varphi }\left( b-a\right) }\int\limits_{a}^{a+e^{i\varphi
}\left( b-a\right) }f\left( x\right) g\left( x\right) dx  \notag \\
&&  \label{160} \\
&\leq &\left( \frac{rs}{\left( r+2\right) \left( s+2\right) }\right) ^{\frac{%
1}{2}}\left( \left[ f\left( a\right) \right] ^{r}+\left[ f\left( b\right) %
\right] ^{r}\right) ^{\frac{2}{r}}\left( \left[ g\left( a\right) \right]
^{s}+\left[ g\left( b\right) \right] ^{s}\right) ^{\frac{2}{s}}.  \notag
\end{eqnarray}
\end{theorem}

\begin{proof}
Since $f$ is $\varphi -r-convex$ function and $g$ is $\varphi -s-convex$
function ($r>0,s>0$), then we have%
\begin{equation}
f\left( a+te^{i\varphi }\left( b-a\right) \right) \leq \left( \left(
1-t\right) \left[ f\left( a\right) \right] ^{r}+t\left[ f\left( b\right) %
\right] ^{r}\right) ^{\frac{1}{r}}  \label{17}
\end{equation}

\begin{equation}
g\left( a+te^{i\varphi }\left( b-a\right) \right) \leq \left( \left(
1-t\right) \left[ f\left( a\right) \right] ^{r}+t\left[ f\left( b\right) %
\right] ^{r}\right) ^{\frac{1}{r}}.  \label{18}
\end{equation}%
Multiplying both sides of (\ref{17}) by (\ref{18}), it follows that%
\begin{eqnarray}
&&f\left( a+te^{i\varphi }\left( b-a\right) \right) g\left( a+te^{i\varphi
}\left( b-a\right) \right)   \label{19} \\
&\leq &\left( \left( 1-t\right) \left[ f\left( a\right) \right] ^{r}+t\left[
f\left( b\right) \right] ^{r}\right) ^{\frac{1}{r}}\left( \left( 1-t\right) %
\left[ g\left( a\right) \right] ^{s}+t\left[ g\left( b\right) \right]
^{s}\right) ^{\frac{1}{s}}.  \notag
\end{eqnarray}%
Integrating the inequality (\ref{19}) with respect to $t$ over $\left[ 0,1%
\right] $,\ we obtain%
\begin{eqnarray}
&&\frac{1}{e^{i\varphi }\left( b-a\right) }\int\limits_{a}^{a+e^{i\varphi
}\left( b-a\right) }f\left( x\right) g\left( x\right) dx  \label{20} \\
&\leq &\int\limits_{0}^{1}\left( \left( 1-t\right) \left[ f\left( a\right) %
\right] ^{r}+t\left[ f\left( b\right) \right] ^{r}\right) ^{\frac{1}{r}%
}\left( \left( 1-t\right) \left[ g\left( a\right) \right] ^{s}+t\left[
g\left( b\right) \right] ^{s}\right) ^{\frac{1}{s}}dt.  \notag
\end{eqnarray}%
Using Cauchy Swartz's inequality, we have%
\begin{eqnarray}
&&\int\limits_{0}^{1}\left( \left( 1-t\right) \left[ f\left( a\right) \right]
^{r}+t\left[ f\left( b\right) \right] ^{r}\right) ^{\frac{1}{r}}\left(
\left( 1-t\right) \left[ g\left( a\right) \right] ^{s}+t\left[ g\left(
b\right) \right] ^{s}\right) ^{\frac{1}{s}}dt  \label{21} \\
&\leq &\left( \int\limits_{0}^{1}\left( \left( 1-t\right) \left[ f\left(
a\right) \right] ^{r}+t\left[ f\left( b\right) \right] ^{r}\right) ^{\frac{2%
}{r}}dt\right) ^{\frac{1}{2}}\left( \int\limits_{0}^{1}\left( \left(
1-t\right) \left[ g\left( a\right) \right] ^{s}+t\left[ g\left( b\right) %
\right] ^{s}\right) ^{\frac{1}{s}}dt\right) ^{\frac{1}{2}}  \notag
\end{eqnarray}%
Using Young's inequality($2ab\leq a^{2}+b^{2}$) for right-hand side of the
inequality (\ref{21}), we have%
\begin{eqnarray}
&&\left( \int\limits_{0}^{1}\left( \left( 1-t\right) \left[ f\left( a\right) %
\right] ^{r}+t\left[ f\left( b\right) \right] ^{r}\right) ^{\frac{2}{r}%
}dt\right) ^{\frac{1}{2}}\left( \int\limits_{0}^{1}\left( \left( 1-t\right) %
\left[ g\left( a\right) \right] ^{s}+t\left[ g\left( b\right) \right]
^{s}\right) ^{\frac{1}{s}}dt\right) ^{\frac{1}{2}}  \label{22} \\
&\leq &\frac{1}{2}\int\limits_{0}^{1}\left( \left( 1-t\right) \left[ f\left(
a\right) \right] ^{r}+t\left[ f\left( b\right) \right] ^{r}\right) ^{\frac{2%
}{r}}dt+\frac{1}{2}\int\limits_{0}^{1}\left( \left( 1-t\right) \left[
g\left( a\right) \right] ^{s}+t\left[ g\left( b\right) \right] ^{s}\right) ^{%
\frac{2}{s}}dt.  \notag
\end{eqnarray}%
Using Minkowski's inequality right-hand side of the inequality (\ref{22}),
we have%
\begin{eqnarray}
&&\int\limits_{0}^{1}\left( \left( 1-t\right) \left[ f\left( a\right) \right]
^{r}+t\left[ f\left( b\right) \right] ^{r}\right) ^{\frac{2}{r}}dt
\label{23} \\
&\leq &\left[ \left( \int\limits_{0}^{1}\left( 1-t\right) ^{\frac{2}{r}}%
\left[ f\left( a\right) \right] ^{2}dt\right) ^{\frac{r}{2}}+\left(
\int\limits_{0}^{1}t^{\frac{2}{r}}\left[ f\left( b\right) \right]
^{2}dt\right) ^{\frac{r}{2}}\right] ^{\frac{2}{r}}  \notag \\
&=&\left( \frac{r}{r+2}\right) \left( \left[ f\left( a\right) \right] ^{r}+%
\left[ f\left( b\right) \right] ^{r}\right) ^{\frac{2}{r}}.  \notag
\end{eqnarray}%
Similarly we have:%
\begin{equation}
\int\limits_{0}^{1}\left( \left( 1-t\right) \left[ g\left( a\right) \right]
^{s}+t\left[ g\left( b\right) \right] ^{s}\right) ^{\frac{2}{s}}dt\leq
\left( \frac{s}{s+2}\right) \left( \left[ g\left( a\right) \right] ^{s}+%
\left[ g\left( b\right) \right] ^{s}\right) ^{\frac{2}{s}}.  \label{240}
\end{equation}%
Adding (\ref{23}) and (\ref{240}) and rewriting (\ref{20}), we obtain (\ref%
{16}).

Now, using Minkowski's inequality for right-hand side of the inequality (\ref%
{21}), we have%
\begin{eqnarray}
&&\left( \int\limits_{0}^{1}\left( \left( 1-t\right) \left[ f\left( a\right) %
\right] ^{r}+t\left[ f\left( b\right) \right] ^{r}\right) ^{\frac{2}{r}%
}dt\right) ^{\frac{1}{2}}  \label{25} \\
&\leq &\left[ \left( \int\limits_{0}^{1}\left( 1-t\right) ^{\frac{2}{r}}%
\left[ f\left( a\right) \right] ^{2}dt\right) ^{\frac{r}{2}}+\left(
\int\limits_{0}^{1}t^{\frac{2}{r}}\left[ f\left( b\right) \right]
^{2}dt\right) ^{\frac{r}{2}}\right] ^{\frac{1}{r}}  \notag \\
&=&\left( \frac{r}{r+2}\right) ^{\frac{1}{2}}\left( \left[ f\left( a\right) %
\right] ^{r}+\left[ f\left( b\right) \right] ^{r}\right) ^{\frac{1}{r}}, 
\notag
\end{eqnarray}%
and similarly%
\begin{equation}
\left( \int\limits_{0}^{1}\left( \left( 1-t\right) \left[ g\left( a\right) %
\right] ^{s}+t\left[ g\left( b\right) \right] ^{s}\right) ^{\frac{1}{s}%
}dt\right) ^{\frac{1}{2}}\leq \left( \frac{s}{s+2}\right) ^{\frac{1}{2}%
}\left( \left[ g\left( a\right) \right] ^{s}+\left[ g\left( b\right) \right]
^{s}\right) ^{\frac{1}{s}}.  \label{26}
\end{equation}%
Writing (\ref{25}) and (\ref{26}) in (\ref{21}), and rewriting (\ref{20}),
we get the desired inequality in (\ref{160}). The proof is complete.
\end{proof}

\begin{corollary}
\label{c2} Under the asumptions of Theorem \ref{tt4} and with $s=r=1$ we have%
\begin{equation*}
\frac{1}{e^{i\varphi }\left( b-a\right) }\int\limits_{a}^{a+e^{i\varphi
}\left( b-a\right) }f\left( x\right) g\left( x\right) dx\leq \frac{\left( %
\left[ f\left( a\right) \right] +\left[ f\left( b\right) \right] \right)
+\left( \left[ g\left( a\right) \right] +\left[ g\left( b\right) \right]
\right) }{6}
\end{equation*}%
and%
\begin{equation*}
\frac{1}{e^{i\varphi }\left( b-a\right) }\int\limits_{a}^{a+e^{i\varphi
}\left( b-a\right) }f\left( x\right) g\left( x\right) dx\leq \frac{\left( %
\left[ f\left( a\right) \right] +\left[ f\left( b\right) \right] \right)
^{2}\left( \left[ g\left( a\right) \right] +\left[ g\left( b\right) \right]
\right) ^{2}}{3}.
\end{equation*}
\end{corollary}

\begin{corollary}
\label{c3} Under the asumptions of Theorem \ref{tt4} and with $s=r$ and $%
f\left( x\right) =g\left( x\right) $, we have%
\begin{equation*}
\frac{1}{e^{i\varphi }\left( b-a\right) }\int\limits_{a}^{a+e^{i\varphi
}\left( b-a\right) }f^{2}\left( x\right) dx\leq \left( \frac{r}{r+2}\right)
\left( \left[ f\left( a\right) \right] ^{r}+\left[ f\left( b\right) \right]
^{r}\right) ^{\frac{2}{r}}
\end{equation*}%
and%
\begin{equation*}
\frac{1}{e^{i\varphi }\left( b-a\right) }\int\limits_{a}^{a+e^{i\varphi
}\left( b-a\right) }f^{2}\left( x\right) dx\leq \left( \frac{r}{r+2}\right)
\left( \left[ f\left( a\right) \right] ^{r}+\left[ f\left( b\right) \right]
^{r}\right) ^{\frac{4}{r}}.
\end{equation*}
\end{corollary}

\begin{remark}
If we take $g\left( x\right) =1$ in Corollary \ref{c2} we have%
\begin{equation*}
\frac{1}{e^{i\varphi }\left( b-a\right) }\int\limits_{a}^{a+e^{i\varphi
}\left( b-a\right) }f\left( x\right) dx\leq \frac{\left( \left[ f\left(
a\right) \right] +\left[ f\left( b\right) \right] \right) +2}{6}
\end{equation*}%
and%
\begin{equation*}
\frac{1}{e^{i\varphi }\left( b-a\right) }\int\limits_{a}^{a+e^{i\varphi
}\left( b-a\right) }f\left( x\right) g\left( x\right) dx\leq \frac{4\left( %
\left[ f\left( a\right) \right] +\left[ f\left( b\right) \right] \right) ^{2}%
}{3}.
\end{equation*}
\end{remark}

\end{document}